\documentclass[reqno]{amsart}

\usepackage{amsmath}
\usepackage{amsthm}
\usepackage{amsfonts}
\usepackage{amssymb}
\usepackage{url}
\usepackage{enumerate}
\usepackage[pdftex,bookmarks=true]{hyperref}

\usepackage{xcolor}

\newcommand\R{{\mathbb{R}}}

\newcommand\Z{{\mathbf{Z}}}

\newcommand\eps{{\varepsilon}}

\newcommand\BZ{{\mathbf Z}}

\theoremstyle{plain}
 \newtheorem{theorem}{Theorem}[section]
 
 \newtheorem{question}[theorem]{Question}

 \newtheorem{proposition}[theorem]{Proposition}

 \newtheorem{corollary}[theorem]{Corollary}

\newtheorem{remark}[theorem]{Remark}

\theoremstyle{definition}

\include{psfig}

\begin{document}

\title[Sum-free sets in  groups]{Sum-free sets in  groups: a survey}

\author{Terence Tao}
\address{Department of Mathematics, UCLA, Los Angeles CA 90095-1555}
\email{tao@@math.ucla.edu}
\thanks{T. Tao is supported by NSF grant DMS-0649473 and by a Simons Investigator Award.}

\author{Van Vu}
\thanks{V. Vu is supported by research grants DMS-0901216 and AFOSAR-FA-9550-09-1-0167.}
\address{Department of Mathematics, Yale University, New Haven, CT 06520, USA}
\email{van.vu@yale.edu}

\begin{abstract} We discuss several questions concerning sum-free sets in groups, raised by 
Erd\H{o}s in his survey  {\it ``Extremal problems in number theory'' } (Proceedings of the Symp. Pure Math. VIII AMS)  published in  1965. 

Among other things, we give a characterization for large sets $A$ in an abelian group $G$  which do not contain a subset $B$ of  fixed size $k$ such that the sum of any two different elements of $B$ do not belong to $A$ (in other words, $B$ is sum-free with respect to 
$A$). Erd\H{o}s, in the above mentioned survey, conjectured that  if $|A|$ is sufficiently large compared to $k$, then $A$ contains two elements that add up to zero. This is known to be true for $k \leq 3$. We give counterexamples for 
all $k \ge 4$.  On the other hand, using the new  characterization result, we are able to prove  a positive result  in the case when $|G|$ is not divisible by small primes.
\end{abstract}

\maketitle

\setcounter{tocdepth}{2}

\centerline{\emph{To Ron Graham for his $80^{\operatorname{th}}$ birthday}}
\section {Sum-free set problems } 

In this note, we discuss a few  questions of Erd\H{o}s, which appeared  in his survey {\it ``Extremal problems in number theory''} \cite{Erd65}, published almost half a century ago. 
The survey started with  his study with Moser on Littlewood-Offord type anti-concentration results. However, as he pointed out, this study led to several 
questions concerning  sum-free sets, which are interesting in their own right. 

In what follows, for a finite set $A$ of an additive group $G$, we set $$2A := A+A  := \{ a_1 + a_2 | a_i \in A \}$$ and $$ 2^{\ast }A = A   \overset{\ast}+ A:= \{ a_1 + a_2| x_i \in A , a_1 \neq a_2 \}. $$
We first consider  the case when $G = \R$, the set of real numbers.

 For a set  $A $ of  real numbers, let $f(A)$ be the size of the  largest subset $B$ of $A$ 
 such that $(B + B) \cap B = \emptyset $.   Define $f(n):= \min_{A, |A| = n }  f(A)$. 

Using a simple, but beautiful, probabilistic argument, Erd\H{o}s \cite{Erd65} proved that $f(n) \ge \frac{n}{3} $. 
Let $T$ be a large number and $I_j$ be the collection of $\alpha \in [0,T]$ such that  $a_j \alpha \,\, (\operatorname{mod} \,\,1)  \in (1/3,2/3)$. It is easy to see 
that  there is a number $C$ which may depend on the $a_j$'s, but is independent of $T$ such that 

$$\left| \mu (I_j) - \frac{T}{3} \right|  \le C, $$ for all $1 \le j \le n$. By the pigeon hole principle,  there is a number $\alpha$ such that there are at least $n/3$ indices $j$ satisfying 
$a_j \alpha \,\, (\operatorname{mod} \,\,1)  \in (1/3,2/3)$. These form the desired set $B$.

It is  surprisingly hard to 
improve upon this bound. Alon and Kleitman \cite{AK} modified  Erd\H{o}s'  argument slightly to have 
$f(n) \ge  \frac{n+1}{3}$. Bourgain \cite{B},  using a much more sophisticated  and entirely different  argument, proved  $f(n) \ge \frac{n+2}{3} $.  This is still the best 
lower bound to date.

From above, a   recent breakthrough by  Eberhard, Green, and Manners \cite{Green}  showed that the constant $1/3$ cannot be improved,  namely $f(n) \le (\frac{1}{3} +o(1))  n$ (see  Eberhard's paper \cite{Eber} for a generalization). It is a fascinating open  problem to determine whether 
$f(n) -\frac{n}{3} $ tends to infinity with $n$. 

The next sum-free problem  Erd\H{o}s discussed in  his  survey concerns a stronger notion of sum-freeness.  Given a set $A$, we denote by $\phi (A)$ 
 the size of the largest subset $B$ such that  $(B \overset{\ast}+  B) \cap A = \emptyset $; following \cite{Ruzsa}, we say that $B$ is \emph{sum-avoiding} in $A$ if this occurs. Similarly, we define $\phi(n) := \min_ {A, |A| =n } \phi (A)$. 
(Notice that the problem is easy if one considers $B+B$ instead of $B \overset{\ast}+  B$, as one can take $A =\{1, 2, \dots, 2^{n-1} \}$, which shows $\phi (n) =1$.)

The  problem of determining the order of magnitude of $\phi (n)$ is still wide open, despite efforts from many researchers through a long period. 
In \cite{Erd65}, Erd\H{o}s  mentioned a result of Selfridge  that showed  $\phi (n) \le n/4$  and suggested that it probably has order $o(n)$.
Choi \cite{choi}, using sieve methods, 
proved that $\phi (n) \le  n^{2/5 + o(1) }$. He also noted that in this problem, it suffices to consider the special case when $A$ is a set of integers, which, in modern term, is a corollary
of Freiman's isomorphism. Choi's result was slightly improved by Baltz, Schoen, and Srivastav \cite{baltz}, who showed that $\phi (n)  = O(n^{2/5} \log n )$ .
 In 2005, Ruzsa \cite{Ruzsa} obtained a  more significant  improvement, proving that  $\phi (n) \le \exp (O(\sqrt {\log n }))$.

From below, Choi showed that $\phi (n) \ge \log_2 n$ and Ruzsa improved it slightly to $2  \log_3 n$, which seems to be the limit of greedy constructions (see \cite{Ruzsa, ssv} for more details). 
One can prove a less precise  bound $\phi (n) \ge \ln n +O(1)   $ using Tur\'an theorem from extremal graph theory.  Without loss of generality, one can assume that $n/2$ elements of $A$ are positive, say 
$0 < a_1 \le \dots \le a_{n/2} $.  Define a graph on this set by  connecting $a_i$ with $a_j$ if 
$a_i +a_ j \in A$. It is clear that the degree of $a_j$ is at most $d_j:= n/2 -j$.  On the other hand, Tur\'an's theorem asserts that 
the independence number of a graph is at least the sum of the reciprocal of the positive degrees. This implies that we have  an independent set $B$ of size at least 
$\sum_{j=1}^{n/2} \frac{1}{j} = \ln n +O(1) $. But, by the definition of the graph, $ 2^{\ast} B  \cap  A = \emptyset$.

For sometime, it was speculated that $\Theta (\log n)$ is the right order of magnitude of $\phi (n)$. However, about ten years ago, 
Sudakov, Szemer\'edi and Vu \cite{ssv} managed to push beyond $\log n$ by showing $\frac{\phi  (n)}{\log n} $ tends to infinity. Quantitatively, they proved that $\frac{\phi (n) }{ \log n }  \ge     \log ^{(5)} n $, where $\log^{(5)} n$ is the fifth iterated logarithm of $n$. 
More recently, Dousse \cite{JD} improved the lower bound to $  (\log^{(3) } n)^{1/32772-o(1) }$, and Shao \cite{Shao} improved it further to $(\log ^{(2) } n) ^{1/2-o(1) }$. These works are elaborate and    rely on some powerful tools in additive combinatorics. 
Nevertherless, the gap between the upper bound and the lower bound remains significant.

\section {Sum-free sets in Groups} 

Now we discuss both problems in a more general setting when  $A$ is a subset of an abelian group $G$. In what follows, we focus on the case when $G$ is abelian with finite rank. 
By the characterization of abelian groups, $G$ has the form 

\begin{equation} \label{representation}  {\bf Z} ^r \oplus \BZ/q_1 \BZ \oplus \cdots \oplus \BZ/q_m \BZ  \end{equation} 
where $\BZ/N\BZ$ is the cyclic group of order $N$ and $q_i$ are powers of (not necessarily distinct) prime numbers.  

The existence of torsion changes the nature of both problems significantly.  For the first problem, Alon  and Kleitman \cite{AK}
 showed that every subset $A$ of a group $G$ contains 
a subset $B$ such that $2B \cap A = \emptyset $ and $|B| > \frac{2}{7} |A|$.  This bound is sharp in the strong sense that $\frac{2}{7}  |A|$ cannot be replaced even by $\frac{2}{7} |A|+1$. 
To see this, one can use a  result of Rhemtulla and Street \cite{RS}, which asserts that 
if $G= (\BZ/p \BZ)^s $, where $p=3k+1$ is a prime, then the maximum sum free subset of $G$ has size $k p^{s-1}$. Taking $k=2$ and letting  $s$  tend  to infinity, we obtain the claim.

The obvious open question here is what happens if we focus on a specific group. Using the notation of the previous section, we set 

$$f_G (n) := \min_{A \subset G, |A|= n } f(A). $$

The value of   $f_G(n)$ varies with  $G$. For instance, if we take $G= {\bf Z}_p$, for $p$ much larger than $n$, then Erd\H{o}s's argument for the real case also works here, giving
$f_G(n) \ge \frac{n}{3} $.  Define $ h(G) := \frac{f_G (n) }{n}  $; the solution to the problem is given by the following theorem:

\begin{theorem} \label{GR}
\begin{itemize}
\item[(i)] If $|G|$ is divisible by a prime  $p \equiv 2 (mod \,\, 3) $ then $h(G) = \frac{1}{3} + \frac{1}{3q}$, where $q$ is the smallest such prime.

\item[(ii)] If $|G|$ is not divisible by any prime  $p \equiv 2 (mod \,\, 3) $ and $3 | |G|$, then $h(G) = \frac{1}{3} $. 

\item[(iii)] If $|G|$ is only divisible by primes $p \equiv 1 (mod \,\, 3) $, then $h(G) = \frac{1}{3} -\frac{1}{3m}$, where $m$ is the largest order of any element of $G$. 
\end{itemize}
\end{theorem} 

Parts (i) and (ii) of this theorem were established by Diananada and Yap \cite{DY}; the remaining case (iii) was obtained by Green and Ruzsa \cite{GR}, following some partial results by Yap \cite{Yap1,Yap2} and Rhemtulla-Street \cite{RS2}.

\vskip2mm

The  second problem is more delicate, as we first need to  find the right question to ask. One can follow the 
above discussion and define 

$$\phi_G (n) := \min_{A \subset G, |A|= n } \phi (A). $$

The subtlety in this definition is that $\phi $ is not monotone in $n$. Furthermore, it can be exactly 1  often. If $n$ happens to be the size of a subgroup $H$ of 
$G$, then just take $A$ to be $H$ and we have $\phi (A)=1$. It thus shows that one cannot expect any universal bound like Alon-Kleitman's. On the other hand, it is not too hard to 
show that being a subgroup is essentially the only reason  $\phi (A)=1$.

\begin{proposition}[Characterisation of $\phi(A)=1$]\label{easy} Let $A$ be a finite subset of an additive group $G$.  Then $\phi(A)=1$ if and only if one of the following is true:
\begin{itemize}
\item $A=H$, where $H \leq G$ is a subgroup of $G$.
\item $A=H \backslash \{0\}$, where $H \leq G$ is a $2$-torsion subgroup of $G$ (thus $2x=0$ for all $x \in H$).
\item $A = \{b\}$ for some $b \in G$.
\item $A = \{b,0\}$ for some $b \in G$.
\item $A = \{b,0,-b\}$ for some $b \in G$.
\end{itemize}
\end{proposition}

\begin{proof} It is easy to verify that $\phi(A)=1$ in all of the above five cases.  Now suppose that $\phi(A)=1$; then we have
\begin{equation}\label{sum}
b_1 + b_2 \in A \hbox{ whenever } b_1, b_2 \in A \hbox{ and } b_1 \neq b_2
\end{equation}

Suppose first that there exists a non-zero $a \in A$ such that $2a \in A$.  Then from \eqref{sum} we see that the injective map $x \mapsto x+a$ maps $A$ to $A$, and thus must be a bijection on $A$.  This implies that the map $x \mapsto x-a$ is also a bijection on $A$.  Combining these facts with \eqref{sum}, we see that $A$ is closed under addition (since we can shift $b_1, b_2, b_1+b_2$ back and forth by $a$ as necessary to make $b_1, b_2$ distinct).  Since $A$ is finite, every element must have finite order, and then $A$ is closed under negation, and so $A$ is a subgroup.

It remains to consider the case when 
\begin{equation}\label{2aa}
2a \not \in A \hbox{ for every non-zero } a \in A.
\end{equation}
Suppose now that there exists an element $b \in A$ such that $2b \neq 0$.  We claim that $A$ must then lie in the group generated by $b$.  For if this were not the case, then take an element $a \in A$ which is not generated by $b$, in particular $a \neq 0$.  By iterating \eqref{sum} we see that $a+kb \in A$ for all positive $k$, thus $b$ must have finite order.  In particular, $a+b, a-b \in A$, and by \eqref{sum} again (and the hypothesis $2b \neq 0$) we see that $2a \in A$, contradicting \eqref{2aa}.  Once $A$ lies in the group generated by $b$, it is not hard to see that $A$ must be one of $\{b\}$, $\{b,0\}$, or $\{b,0,-b\}$, simply by using the observation from \eqref{sum} that the map $x \mapsto x+b$ maps $A \backslash \{b\}$ into $A$, together with \eqref{2aa}.

The only remaining case is the $2$-torsion case when $2b=0$ for all $b \in A$.  Then either $A = \{0\}$, or else by \eqref{2aa} $A$ does not contain zero.  In the latter case we observe from \eqref{sum} that $A \cup \{0\}$ is closed under addition and is thus a $2$-torsion group.  The claim follows.
\end{proof}

The problem we would like to pursue  here is to obtain a similar classification for sets  $A$ with $\phi (A)= k$, for any given $k$.  The main theorem is the content of the next section.

\section{Structure of a set $A$ with $\phi (A) \le k$} 

For a given $k$, we can guarantee $\phi (A) \le k$ by taking $A$ to be the union of $k$ subgroups. 
More generally, we can take $A$ be the union of $k- l$ subgroups and a set of $l$ elements, for any $ 0 \le l \le k$. 

The main  new result we would like to announce is a 
 partial converse to this observation. Roughly speaking, we prove that if $\phi (A) \le k$, then $A$ is the union of $k-l$ dense subsets of subgroups and a set of bounded  size.  

\begin{theorem}[Small $\phi$ implies covering by groups]\label{main} Let $A$ be a finite subset of an additive group 
$G$ with $\phi(A) \leq k$ for some $k \geq 1$.  Then there exist finite subgroups $H_1,\ldots,H_{m}$ of $G$ with $0 \leq m \leq k$
 such that 
\begin{equation}\label{amm}
|A \backslash (H_1 \cup \ldots \cup H_{m})| \leq C(k)
\end{equation}
and 
\begin{equation}\label{amm-2}
|A \cap H_i| \geq |H_i|/C(k)
\end{equation}
for all $1 \leq i \leq m$.  Here $C(k) > 0$ is a quantity  depending only on $k$ (in particular, it does not depend on $G$ or $|A|$).  If furthermore $m=k$, we may strengthen \eqref{amm} to
$$ A \subseteq H_1 \cup \ldots \cup H_k.$$
\end{theorem}

Note that Proposition \ref{easy} gives the $k=1$ case of this theorem with $C(1)=3$. The formulation of the theorem was motivated by Freiman type inverse theorems in additive combinatorics. The full proof of this theorem is long and fairly technical and will be presented in a coming paper \cite{TVlongversion}, but we will discuss some key ideas in
Section \ref{key}. Due to the complexity of the argument, we choose to present the proof using non-standard analysis. One byproduct of this is that our arguments currently provide no bound whatsoever on the quantity $C(k)$ appearing in the above theorem; 
we expect that if one were to translate the nonstandard analysis arguments back to a standard finitary setting, that the bound obtained on $C(k)$ would be of Ackermann type in $k$ or worse.

 Theorem \ref{main} does not  describe the structure of $A$ inside each of the component groups $H_i$, other than to establish positive density in the sense of \eqref{amm-2}. However, we should warn the readers that 
 one does not expect as simple a description of the sets $A \cap H_i$ as in Proposition \ref{easy}. For instance, take  $A$ to be  the union of a finite group $H$ and an arbitrary subset of 
 a coset $x+H$ with $2x \in H$, then $\phi (A) \le 2$, but $A$ {\it is not}  the union of 
 two subgroups or the union of one subgroup and a finite set.  On the other hand, in some special cases, we can obtain a stronger statement that pushes the density $ \frac{ | A \cap H_i | }{ |H_i | }$ close to 1; see Theorem \ref{main2} below. 

\section{Erdos's zero-sum problem} 

While discussing $\phi (A)$,  Erd\H{o}s  \cite{Erd65} raised  the following question. 

\begin{question}\label{quest}  Let $k$ be a natural number, let $G$ be a finite additive group, and let $A$ be a subset of $G$ with $\phi(A) < k$.  
Assume that $|A|$ is sufficiently large depending on $k$.  Does there necessarily exist $a_1,a_2 \in A$ such that $a_1+a_2=0$?
\end{question}

It is easy to see that the answer is affirmative when $k=2$ (by Proposition \ref{easy}, for instance). The same answer holds for the case $k=3$, as verified by  Luczak and Schoen \cite{ls} in 1996.
On the other hand,  for every $k \geq 4$, we found a simple counterexample 

\begin{proposition}[Counterexample for  $k \ge 5$]\label{counter}  Let $n \ge 4$ be a natural number, and set $G$ to be the cyclic group $G := \Z/2^n \Z$.  Let $A \subset G$ be the set
$$ A := \{ (4m+1)2^j \hbox{ mod } 2^n: m \in \Z, 0 \leq j \leq n-2 \}.$$
Thus, for instance, if $n=4$, then $A = \{ 1, 2, 4, 5, 9, 10, 13 \hbox{ mod } 16 \}$. 
Then $\phi(A) = 4$ and $|A| = 2^{n-1}-1$, but there does not exist $a_1,a_2 \in A$ with $a_1+a_2=0$.
\end{proposition}


\begin{proof}  It is easy to see that if $a \in A$, then $-a \not \in A$, and that
$$ |A| = 2^{n-2} + 2^{n-3} + \dots + 1 = 2^{n-1}-1$$
as claimed, and the set $\{ 1, 2, 5, 10 \hbox{ mod } 2^n \}$ is always sum-avoiding in $A$, so $\phi(A) \geq 4$.  The only remaining thing to establish is the upper bound $\phi(A) \leq 4$.  Suppose for contradiction that there existed distinct $a_1,a_2,a_3,a_4,a_5 \in A$ such that the $\binom{5}{2}$ sums $a_i+a_{i'}$ with $1 \leq i < i' \leq 5$ were all outside $A$.  We can write $a_i = (4m_i+1)2^{j_i} \hbox{ mod } 2^n$ with $j_1 \leq \dots \leq j_5$.  If $j_5 > j_1+1$ then $a_5$ is a multiple of $4 \times 2^{j_1}$, and hence $a_1+a_5 = (4m_1+1)2^{j_1}+a_5$ lies in $A$, a contradiction.  Thus $j_1,\dots,j_5$ lie in $\{j_1,j_1+1\}$.  By the pigeonhole principle, we can then find $1 \leq i < i' \leq 5$ such that $j_i = j_{i'}$ and such that $m_i, m_{i'}$ have the same parity.  Note that $j_i=j_{i'}$ cannot equal $n-2$ since $a_i,a_{i'}$ would then both equal $2^{n-2} \hbox{ mod } 2^n$.  But then $a_i+a_{i'}$ is of the form $(4m+1)2^{j_i+1}$, where $m$ is the average of $m_i$ and $m_{i'}$, and so $a_i+a_{i'} \in A$, again a contradiction.
\end{proof}

\begin{proposition} [Counterexample for $k=4$]\label{k4}  Let $G := (\Z/7\Z) \times H$ for some arbitrary finite group $H$, and let $A := A_0\times H$ where $A_0 := \{ 1 \hbox{ mod } 7, 2 \hbox{ mod } 7, 4 \hbox{ mod } 7 \}$.  Then $\phi(A)=3$ and $|A| = 3|H|$, but there does not exist $a_1,a_2 \in A$ with $a_1+a_2=0$.
\end{proposition}

\begin{proof}  The only non-trivial claim is that $\phi(A)=3$.  From computing sums from the set $A_0 \times \{0\}$ we see that $\phi(A) \geq 3$.  Suppose for contradiction that there existed distinct $a_1,a_2,a_3,a_4 \in A$ with all sums $a_i+a_j$ outside $A$.  By the pigeonhole principle we can find $1 \leq i < j \leq 4$ such that $a_i, a_j \in \{a\} \times H$ for some $a \in A_0$.  Then $a_i+a_j \in \{2a\} \times H$.  Since $2a$ is also in $A_0$, we obtain a contradiction.
\end{proof}

\begin{remark} \label{M}
It is clear that one  can use the construction in Proposition \ref{k4}  for all values $k$ where $2^k-1$ is a prime (Mersenne primes). 
\end{remark}

Thus Erd\H{o}s's question is now resolved for all values of $k$.  However, with Theorem \ref{main} in hand, one still has  the strong feeling that the answer must be {\it morally} affirmative. After all, if $A$ is a dense subset of  a group $H$, then typically 0 (or any element of $H$, for that matter) must be 
represented as the sum of two elements of $A$ a large number of times.  The above counterexamples rely
 heavily on the order of the group $G$ being divisible by a small prime ($2$ and $7$ respectively).  Indeed, in  the opposite case, where the order of $G$ is not divisible by any small primes,  we are able to prove a positive result. 
 As a matter of fact, in this case we have the following strengthening of Theorem \ref{main}.

 \begin{theorem} \label{main2} 
Let  $k$ be a natural number and $\eps$  be a positive constant. Suppose $C_0$ is sufficiently large depending on $k,\eps$.  Let
$A$ be a subset of a finite group $G$ with $\phi(A) < k$, $|A| \geq C_0$, and $|G|$ not divisible by any prime less than $C_0$.  Then there exist finite subgroups $H_1,\dots,H_m$ of $G$ with $0 \leq m < k$ such that $|A \cap H_i| > (1-\eps) |H_i|$ for 
every $i=1,\dots,m$ and  $|A \backslash (H_1 \cup \dots \cup H_m)| \leq C_0$.  If $m = k-1$, we can take $A \backslash (H_1 \cup \dots \cup H_m)$ to be empty.
\end{theorem}

If $\eps = 1/2$ and $|A \cap H_i | \ge (1-\eps) |H_i|$, then it is trivial that $A \cap H_i$ contains two elements $a, a'$ which sum up to zero.  Thus, we obtain 

\begin{corollary}[The case of no small prime divisors]\label{odd}   For any fixed $k$ there is a constant $C(k)$ such that the answer to Erd\H{o}s' question is affirmative for all groups $G$ where 
$|G|$ does not have any prime factor less than $C(k)$. 
\end{corollary}

It is an interesting question to classify those groups where the answer to Erd\H{o}s' question   is positive.  The number $C(k)$
that  comes from Theorem \ref{main2} is very large (see the remark after Theorem \ref{main}). On the other hand, if there are infinitely many Mersenne primes, then Remark \ref{M} shows that $C(k)$ should 
be at least exponential in $k$.

\section{Some ideas behind the proof of Theorem \ref{main}  } \label{key}

  The intuition behind the proof is as follows. If $\phi(A)$ is equal to some small natural number $k$ and $A$ is large, then we expect many pairs $a,a'$ in $A$ to sum to another element in $A$.
    Standard tools in additive combinatorics, such as the Balog-Szemer\'edi theorem \cite{balog} and Freiman's theorem in an arbitrary abelian group \cite{gr-4},
     then should show that $A$ contains a large component that is approximately a \emph{coset progression} $H+P$: the Minkowski sum of a finite group 
     $H$ and a multidimensional arithmetic progression $P$.  Because of bounds in the real case such as those mentioned in the first section that show that $\phi$ becomes large on large torsion-free sets, 
     one expects to be able to eliminate the role of the ``torsion-free'' component $P$ of the coset progression $H+P$, to conclude that $A$ has large intersection with a finite subgroup $H$. 
      In view of the subadditivity $\phi (A \cup B) \le \phi(A) + \phi (B)$ , one heuristically expects $\phi(A)$ to drop from $k$ to $k-1$ after removing $H$ (that is to say, one expects $\phi(A \backslash H) = k-1$), at 
      which point one can conclude by induction on $k$ starting with Proposition \ref{easy} as a base case.  More realistically, one expects to have to replace the conclusion $\phi(A \backslash H) = k-1$
       with some more technical conclusion that is not exactly of the same form as the hypothesis $\phi(A)=k$, which makes a direct induction on $k$ difficult; instead, 
       one should expect to have to perform a $k$-fold iteration argument in which one removes up to $k$ subgroups $H_1,\dots,H_m$ from $A$ in turn until one is left with a 
       small residual set $A \backslash (H_1 \cup \dots \cup H_m)$. 

Unfortunately, when the group $G$ contains a lot of torsion, removing a large subgroup $H$ from $A$ can leave one with a residual set with no good additive structure, 
and in particular with no bounds whatsoever on $\phi(A \backslash H)$.  For instance, suppose that there is an element $x$ of $G \backslash H$ with $2x \in H$, and take $A
$ to be the union of $H$ and an \emph{arbitrary} subset of $x+H$.  Then it is easy to see that $\phi(A)$ is at most $2$, but upon removing the large finite group $H$ from $A$ one is left with an arbitrary subset of $x+H$, and in particular $\phi(A \backslash H)$ can be arbitrarily large.

The problem in this example is that the group $H$ is the ``incorrect'' group to try to remove from $A$; one should instead remove the larger group $H' := H + \{0,x\}$, which contains $H$ as an index two subgroup.  The main difficulty in the argument is then to find an algorithm to enlarge an ``incorrect'' group $H$ to a ``correct'' group that absorbs all the relevant ``torsion'' that is present.  This is not too difficult at the start of the iterative argument mentioned above, but becomes remarkably complicated in the middle of the iteration when one has already removed some number of large subgroups $H_1,\dots,H_{m'}$ from the initial set $A$.  A particular technical difficulty comes from the fact that the groups $H_1,\dots,H_{m'}$, as well as the residual set $A \backslash (H_1 \cup \dots \cup H_{m'})$, can have wildly different sizes; in particular, sets which are negligible when compared against one of the $H_i$, could be extremely large when compared against the residual set $A \backslash (H_1 \cup \dots \cup H_{m'})$.  To get around these issues, one needs to ensure some ``transversality'' between these components of $A$, in the sense that the intersection between any of these two sets (or translates thereof) are much smaller than either of the two sets.  This adds an extra layer of complexity to the iterative argument; so much so, in fact, that it becomes very unwieldy to run the argument in a purely finitary fashion.  Instead, we were forced to formulate the argument in the language of nonstandard analysis\footnote{For some prior uses of nonstandard analysis in additive combinatorics, see e.g. \cite{jin} or \cite{gtz}.} in order to avoid a large number of iterative arguments to manage a large hierarchy of parameters (somewhat comparable in complexity to those used to prove the hypergraph regularity lemma, see e.g. \cite{gowers-hypergraph}, \cite{nagle-rodl-schacht}, \cite{nagle-rodl-schacht1}, \cite{tao:hyper}).  As mentioned before,  a byproduct  of this is that our arguments currently provide no bound whatsoever on the quantity $C(k)$ appearing in the above theorem; indeed we expect that if one were to translate the nonstandard analysis arguments back to a standard finitary setting, that the bound obtained on $C(k)$ would be of Ackermann type in $k$ or worse.

\vskip2mm

{\it Ackowledgement.} We would like to thank B. Green for pointing out several references.

\end{document}